\documentclass[11pt]{article}
\usepackage[numbers,sort&compress]{natbib}
\usepackage{enumerate}
\usepackage{amscd}
\usepackage{amsmath}
\usepackage{latexsym}
\usepackage{amsfonts}
\usepackage{amssymb}
\usepackage{amsthm}
\usepackage{verbatim}
\usepackage{mathrsfs}
\usepackage{enumerate}
\usepackage{hyperref}

\theoremstyle{plain}\newtheorem{definition}{Definition}[section]
\theoremstyle{definition}\newtheorem{theorem}{Theorem}[section]
\theoremstyle{plain}\newtheorem{lemma}[theorem]{Lemma}
\theoremstyle{plain}\newtheorem{coro}[theorem]{Corollary}
\theoremstyle{plain}
\theoremstyle{remark}\newtheorem{remark}{Remark}[section]
\usepackage{xcolor}

\newcommand{\Div}{\mathrm{div}\,}
\newcommand{\B}{\Big}

\newcommand{\be}{\begin{equation}}
\newcommand{\ee}{\end{equation}}
 \newcommand{\ba}{\begin{aligned}}
 \newcommand{\ea}{\end{aligned}}

  \newcommand{\f}{\frac}
    
  \newcommand{\ben}{\begin{enumerate}}
   \newcommand{\een}{\end{enumerate}}

\newcommand{\Rmnum}[1]{\expandafter\@slowromancap\romannumeral #1@}

\allowdisplaybreaks

\numberwithin{equation}{section}
\begin{document}
\title{The role of density  in the  energy conservation  for      the isentropic compressible
Euler equations}
\author{Yanqing Wang\footnote{   College of Mathematics and   Information Science, Zhengzhou University of Light Industry, Zhengzhou, Henan  450002,  P. R. China Email: wangyanqing20056@gmail.com}  ~   \,   Yulin Ye\footnote{Corresponding author. School of Mathematics and Statistics,
Henan University,
Kaifeng, 475004,
P. R. China. Email: ylye@vip.henu.edu.cn} ~   \, and \, Huan Yu \footnote{ School of Applied Science, Beijing Information Science and Technology University, Beijing, 100192, P. R. China Email:  yuhuandreamer@163.com}
 }
\date{}
\maketitle
\begin{abstract}
In this paper, we  study     Onsager's conjecture on the energy conservation for
 the isentropic compressible Euler equations via establishing
  the    energy conservation   criterion  involving the density $\varrho\in L^{k}(0,T;L^{l}(\mathbb{T}^{d}))$.  The motivation is to  analysis  the role of   the integrability of    density
   of  the weak solutions keeping energy  in   this system, since almost all known corresponding  results  require $\varrho\in L^{\infty}(0,T;L^{\infty}(\mathbb{T}^{d}))$.
Our results     imply  that the lower integrability of the density  $\varrho$ means that more integrability of the velocity  $v$ are necessary in energy conservation  and the inverse is also true. The proof relies on the  Constantin-E-Titi type and Lions type
commutators on mollifying kernel.

  \end{abstract}
\section{Introduction}
\label{intro}
\setcounter{section}{1}\setcounter{equation}{0}
In 1949,  Onsager  \cite{[Onsager]}  conjectured that
  every weak solution of the incompressible homogeneous Euler equations with H\"older continuity exponent  $\alpha>1/3$  must conserve     energy and  there exists a weak solution  with  energy dissipation to the Euler equations  when $\alpha<1/3$.  Let  $v$
represent the fluid  velocity field and $\pi$ stand  for  the scalar pressure.
The classical  incompressible Euler equations   describing  inviscid fluids  read
  \be\left\{\ba\label{Euler}
  &  v _{t} +\Div(  v\otimes v)+\nabla
\pi=0,\\
&\Div v=0,\\
& v(x,0)=v_{0}.
\ea\right.\ee

The positive  part of Onsager's conjecture was proved by Constantin-E-Titi  in \cite{[CET]},
where they showed that   energy is  conserved for a weak solution $v$   in the Besov space $L^{3}(0,T; B^{\alpha}_{3,\infty}(\mathbb{T}^{3}))$ with $\alpha>1/3$.  For $\alpha=1/3$, it is shown that every weak solution $v$ of the Euler equations
on $[0, T ]$ conserves energy provided   $v \in L^{3}(0,T; B^{1/3}_{3,c(N)})$ by
Cheskidov-Constantin-Friedlander-Shvydkoy in \cite{[CCFS]}, where
$
B^{1/3}_{3,c(N)}=\{(u\in B^{1/3}_{3,\infty}, \lim_{q\rightarrow\infty}\int_{0}^{T}2^{q}\|\Delta_{q}u\|^{3}_{L^{3}}dt=0)\}$.
 After  a series of papers by De Lellis and  Szekelyhidi \cite{[DS0],[DS1],[DS2],[DS3]},  Isett \cite{[Isett]} successfully solved the second part of Onsager's conjecture in three-dimensional space.
 At this stage,  the progress of  Onsager's conjecture of the 3D incompressible Euler equations  \eqref{Euler} is
satisfactory.

 Note that the density $\varrho$ of the flow in Euler equations  \eqref{Euler} is a constant. A natural extension of Onsager's conjecture is to consider   critical regularity of   conservation of energy in the density-dependent (nonhomogeneous) Euler system and the compressible Euler equations.
 Important progress involving   generalized  Onsager's conjecture in both equations has been made (see \cite{[CY],[NNT1],[ADSW],[EGSW],[LS],[DE]} and   references therein).
Before we state these results, we recall the  nonhomogeneous  Euler system \eqref{NEuler}
\be\left\{\ba\label{NEuler}
&\varrho_t+\nabla \cdot (\varrho v)=0, \\
&(\varrho v)_{t} +\Div(\varrho v\otimes v)+\nabla
\pi=0, \\
&\Div v=0,
 \ea\right.\ee
 and the compressible Euler equations  \eqref{CEuler},
\be\left\{\ba\label{CEuler}
&\varrho_t+\nabla \cdot (\varrho v)=0, \\
&(\varrho v)_{t} +\Div(\varrho v\otimes v)+\nabla
\pi(\varrho )=0,\\
&\pi(\varrho)=\kappa\varrho^{\gamma}, \gamma>1, \kappa=\f{(\gamma-1)^{2}}{4\gamma},
\ea\right.\ee
  In general, one complements equations \eqref{NEuler} or  \eqref{CEuler}  with initial data
\begin{equation}\label{INS1}
	\varrho(0,x)=\varrho_0(x),\ (\varrho v)(0,x)=(\varrho_0 v_0)(x),\ x\in \Omega,
\end{equation}
where we define $v_0 =0$ on the sets $\{x\in \Omega:\ \varrho_0=0\}.$ In the present paper, we consider the periodic case, which means $\Omega=\mathbb{T}^d$ with dimension $d\geq 2$.

In particular, for the density-dependent  Euler equations   \eqref{NEuler},  by means of Constantin-E-Titi type
commutators on mollifying kernel,   Feireisl-Gwiazda-Gwiazda-Wiedemann \cite{[EGSW]} first  established the following  energy conservation criterion of weak solutions: if  the weak solutions satisfy
$$
v \in B_{p, \infty}^{\alpha} ((0, T) \times \mathbb{T}^{d} ), \quad \varrho, \varrho v \in B_{\f{p}{p-2}, \infty}^{\beta} ((0, T) \times \mathbb{T}^{d} ), \quad \pi \in L_{l o c}^{\f{p}{p-1}} ((0, T) \times \mathbb{T}^{d} ),
$$
with  $2 \alpha+\beta>1$ and $0 \leqq \alpha, \beta \leqq 1$,
then the energy is locally conserved. Introducing the  Lions type
commutators on mollifying kernel,    Chen and Yu \cite{[CY]} derived the sufficient conditions for
 weak solutions of the nonhomogeneous  Euler equations keeping energy  as follow:
If there holds
\be\ba\label{cy}
&\varrho \in L^{\infty} ([0, T] \times \mathbb{T}^{d} ) \cap L^{\f{q}{q-3}} (0, T ; W^{1, \f{q}{q-3}} (\mathbb{T}^{d} ) ), \varrho_{t} \in L^{\f{q}{q-3}} ([0, T] \times \mathbb{T}^{d} ),\\
&v\in L^{q} (0, T ; B_{q, \infty}^{\alpha} (\mathbb{T}^{d} ) ), \alpha>1/3.
 \ea\ee
 Then the energy is conserved.

For the isentropic compressible   Euler equations   \eqref{CEuler},
Feireisl-Gwiazda-Gwiazda-Wiedemann   \cite{[EGSW]}  show that
$$\ba
&\quad 0 \leq \underline{\varrho} \leq \varrho \leq \bar{\varrho} \text { a.e. in }(0, T) \times \mathbb{T}^{d},\pi \in C^{2}[\varrho, \bar{\varrho}]\\
&\varrho, \varrho v \in B_{3, \infty}^{\beta} ((0, T) \times \mathbb{T}^{d} ), v \in B_{3, \infty}^{\alpha} ((0, T) \times \mathbb{T}^{d} ),
\ea$$
 and $0 \leq \alpha, \beta \leq 1$ such that
$
\beta>\max  \{1-2 \alpha, \frac{1}{2}(1-\alpha) \}
$   guarantee that the energy of weak solutions is locally conserved.
Very recently, the  hypothesis $\pi\in C^{2}$ was improved to $\pi\in C^{1,\gamma-1}$ with $1\leq\gamma<2$ by Akramov-Debiec-Skipper-Wiedemann  in \cite{[ADSW]}.  Moreover, under the density $\varrho \in L^{\infty} ([0, T] \times \mathbb{T}^{d} ) $,
sufficient conditions    for   weak solutions of system \eqref{NEuler} and \eqref{CEuler} to conserve the energy can be found in \cite{[NNT1],[LS],[BGSTW],[GMS]}.

From the results  mentioned above, almost sufficient conditions involving  non-constant density for persistence of energy requires that
$\varrho\in L^{\infty}(0,T;L^{\infty}(\mathbb{T}^d))$.  This can be derived from $\Div v=0$ in  the nonhomogeneous Euler system. However, for
 compressible case, the
 natural regularity of the density is only
$\varrho\in L^{\infty}(0,T;L^{\gamma}(\mathbb{T}^d))$.
A natural question is that under the hypothesis
$\varrho\in L^{\infty}(0,T;L^{l} (\mathbb{T}^d))$ with $l<\infty$ rather than $\varrho\in L^{\infty}(0,T;L^{\infty}(\mathbb{T}^d))$, what is the minimum regularity to guarantee the energy   balance in isentropic compressible Euler equations \eqref{CEuler}.

We state our first result for   isentropic compressible Euler equations \eqref{CEuler} away from vacuum  as follows.
\begin{theorem}\label{the1.1}
 Let the pair $ (\varrho,v)$ be a  weak solution of isentropic compressible Euler equations  in the sense of definition \ref{wsdefi}. Suppose that   the density $\varrho$ and the velocity $v$ satisfy \be\ba\label{wyy1}
 &0<c\leq \varrho\in L^{\max\{\f{p}{p-3},\f{p(\gamma-1)}{2}\} } (0,T; L^{\max\{\f{q}{q-3},\f{q(\gamma-1)}{2}\} }(\mathbb{T}^{d})),  \\& \nabla  \varrho\in L^{\f{p}{p-3}} (0,T;L^{\f{q}{q-3}}(\mathbb{T}^{d})) ,     \varrho_{t} \in L^{\f{p}{p-3}} (0,T;L^{ \f{q}{q-3}}(\mathbb{T}^{d})),  \\
 &v\in L^{p}(0,T;B^{\alpha}_{q,\infty}(\mathbb{T}^{d})) , \alpha>1/3, p,q>3.
 \ea\ee
 Then the energy is locally conserved, that is,
 \begin{equation}\label{EI}\ba
   \int_{\mathbb{T}^d}\partial_{t}\left( \frac{1}{2}\varrho |v|^2+\kappa\f{\varrho^{\gamma}}{\gamma-1} \right) dx =0,	
\ea
\end{equation}
in the sense of distributions on $(0,T)$.
\end{theorem}

\begin{remark}
By small modification in the proof of Theorem \ref{the1.1}, our result also holds for $p=3$, $q=3$. Then taking $p=q=3$ in \eqref{wyy1}, one derives sufficient conditions for local energy balance
of weak solutions
$$\ba
&0<c\leq\varrho\in L^{\infty} (0,T; L^{\infty }(\mathbb{T}^{d})),   \nabla \varrho\in L^{\infty} (0,T;L^{\infty}(\mathbb{T}^{d})) ,     \varrho_{t} \in L^{\infty} (0,T;L^{ \infty}(\mathbb{T}^{d})),  \\
 &v\in L^{3}(0,T;B^{\alpha}_{3,\infty}(\mathbb{T}^{d})) , \alpha>1/3,
 \ea$$
which is an analogue of classical  Constantin-E-Titi theorem for the Onsager conjecture of the 3D compressible  Euler equations with the density away from vacuum.
\end{remark}
\begin{remark} This theorem implies    that lower integrability of the density  $\varrho$ means that more integrability of the velocity  $v$ are necessary in energy conservation of the  isentropic compressible Euler equations and the inverse is also true.
\end{remark}
\begin{remark}
It seems that this is the first  energy conservation criterion for the compressible Euler equations, which allows the  density to require its upper bound.
\end{remark}
\begin{remark}
All results here can be viewed as addressing  what is the minimum regularity posed on both the density and the velocity to guarantee the energy   balance in isentropic compressible Euler equations \eqref{CEuler}. \end{remark}
\begin{coro}
The energy conservation \eqref{EI} is valid provided   one of the following two conditions is
satisfied
 \begin{enumerate}[(1)]
 \item $ \varrho\in L^{\f{p}{p-3}} (0,T;L^{\f{q}{q-3}}(\mathbb{T}^{d})), \varrho_{t}\in L^{ \f{p}{p-3}  } (0,T; L^{ \f{q}{q-3}  }(\mathbb{T}^{d})),   \nabla\varrho\in L^{\f{p}{p-3}} (0,T;L^{\f{q}{q-3}}(\mathbb{T}^{d})),$
\be\ba
 &v\in L^{p}(0,T;B^{\alpha}_{q,\infty}(\mathbb{T}^d)) , \alpha>1/3,3\leq p,q\leq\f{3\gamma-1}{\gamma-1}.
\ea \ee
 \item $\varrho\in L^{ \f{3\gamma-1}{2}  } (0,T;L^{\f{3\gamma-1}{2}}(\mathbb{T}^{d})),  \varrho_{t}\in L^{ \f{3\gamma-1}{2}  } (0,T; L^{\f{3\gamma-1}{2}}(\mathbb{T}^d))),    \nabla \varrho\in L^{ \f{3\gamma-1}{2}  } (0,T;L^{\f{3\gamma-1}{2}}(\mathbb{T}^{d})),      $
     \be\ba
&v\in L^{\f{3\gamma-1}{\gamma-1}}(0,T;B^{\alpha}_{\f{3\gamma-1}{\gamma-1},\infty}(\mathbb{T}^{d})) , \alpha>1/3.
\ea \ee
\end{enumerate}
\end{coro}
Next, we extend the energy conservation up to the initial time. In this case, the more condition posed on density are needed. The   precise result is the following.
\begin{theorem}\label{the1.3}
 Let the pair $ (\varrho,v)$ be a  weak solution    in the sense of definition \ref{wsdefi}. Assume  that there holds
 \be\ba\label{wyy2}
 &0<c\leq \varrho\in L^{k } (0,T; L^{l}(\mathbb{T}^{d})), \\&  k \geq \max\{\f{p}{p-3},\f{p(\gamma-1)}{2},\frac{(\gamma-1)(d+q)p}{2q-d(p-3)}\}, l\geq \max\{\f{q}{q-3},\f{q(\gamma-1)}{2}\},  \\& \nabla \sqrt{\varrho}\in {L^{\f{2kp}{2k(p-3)-p}}(0,T;L^{\f{2lq}{2l(q-3)-q}}(\mathbb{T}^d))} ,     \partial_t \sqrt{\varrho}\in {L^{\f{2kp}{2k(p-3)-p}}(0,T;L^{\f{2lq}{2l(q-3)-q}}(\mathbb{T}^d))},  \\
 &v\in L^{p}(0,T;B^{\alpha}_{q,\infty}(\mathbb{T}^{d})) , \alpha>1/3, p,q>  3,q> \frac{d(p-3)}{2},\ and\ v_0\in L^{\max\{\frac{2\gamma}{\gamma -1}, \frac{q}{2}\}}(\mathbb{T}^d).
 \ea\ee
 Then the energy  is globally conserved, namely, for any $t\in[0,T]$,
 	\begin{equation}\label{energyeq}
 	\begin{aligned}
 		E(t)  = E(0), 		\end{aligned}\end{equation}
 where $E(t)=\int_{\mathbb{T}^d}\left( \frac{1}{2}\varrho |v|^2+\kappa\f{\varrho^{\gamma}}{\gamma-1} \right) dx$.
\end{theorem}
\begin{remark}
	It should be noted that this result also holds for $p=3$, $q=3$ just by small modification in the proof of Theorem \ref{the1.3}.
\end{remark}

\begin{coro}
Let the weak solution $(\varrho,v)$ of compressible Euler equations meets one of the following condition
 \begin{enumerate}[(1)]
 \item  $0<c\leq \varrho \in L^\infty(0,T;L^{\frac{q}{q-3}})$,\
   $\partial_t\sqrt{\varrho}\in L^{\infty } (0,T; L^{ \f{q}{q-3}  }(\mathbb{T}^{d})),   \nabla \sqrt{\varrho}\in L^{\infty} (0,T;L^{\f{q}{q-3}}(\mathbb{T}^{d})),$
\be\ba
 &v\in L^{3}(0,T;B^{\alpha}_{q,\infty}(\mathbb{T}^d)) , \alpha>1/3,\ 3\leq q\leq\f{3\gamma-1}{\gamma-1}\ and\ v_0\in L^2(\mathbb{T}^d).
\ea \ee
 \item  $0<c\leq \varrho \in L^\infty(0,T;L^{\frac{3\gamma-1}{2}})$,\
 $\partial_t\sqrt{\varrho}\in L^{\infty } (0,T; L^{ \frac{3\gamma-1}{2}  }(\mathbb{T}^{d})),   \nabla \sqrt{\varrho}\in L^{\infty} (0,T;L^{\frac{3\gamma-1}{2}}),$
     \be\ba
&v\in L^{3}(0,T;B^{\alpha}_{\f{3\gamma-1}{\gamma-1},\infty}(\mathbb{T}^{d})) , \alpha>1/3, \ and\ v_0\in L^2(\mathbb{T}^d).
\ea \ee
Then the relation   \eqref{energyeq} is valid.
\end{enumerate}
\end{coro}
\begin{remark}
	Here the  minimum   integrability  of the density in space  direction is $\f{3\gamma-1}{2}$ for energy conservation.
	It is an open problem to derive  energy conservation criterion under the condition
	$\varrho\in L^{\infty}(0,T;L^{\gamma}(\mathbb{T}^d))$.
\end{remark}
Finally, we consider    the energy conservation of the weak solutions of the
isentropic compressible Euler equations
allowing vacuum.
 \begin{theorem}\label{the1.5}
 Let the pair $ (\varrho,v)$ be a solution of  in the sense of definition \ref{wsdefi}. Suppose that there holds
 \be\ba\label{wy1v}
&0\leq \varrho\in L^{k } (0,T; L^{l}(\mathbb{T}^{d})),\\& k \geq \max\{\f{p}{p-3},\f{p(\gamma-1)}{2},\frac{(\gamma-1)(n+q)p}{2q-n(p-3)}\},\ l\geq \max\{\f{q}{q-3},\f{q(\gamma-1)}{2}\},  \\& \nabla \sqrt{\varrho}\in {L^{\f{2kp}{2k(p-3)-p}}(0,T;L^{\f{2lq}{2l(q-3)-q}}(\mathbb{T}^d))} ,     \partial_t \sqrt{\varrho}\in {L^{\f{2kp}{2k(p-3)-p}}(0,T;L^{\f{2lq}{2l(q-3)-q}}(\mathbb{T}^d))},  \\
 &v\in  B^{\beta}_{p,\infty}((0,T);B^{\alpha}_{q,\infty}(\mathbb{T}^{d})) , \beta\geq \alpha>1/3, p,q>  3,q> \frac{d(p-3)}{2},\ and\ v_0\in L^{\max\{\frac{2\gamma}{\gamma -1}, \frac{q}{2}\}}(\mathbb{T}^d).
 \ea\ee
Then the energy  is globally conserved, namely, for any $t\in[0,T]$,
\begin{equation}\label{energyeq-1}
	\begin{aligned}
		E(t)  = E(0), 		\end{aligned}\end{equation}
where $E(t)=\int_{\mathbb{T}^d}\left( \frac{1}{2}\varrho |v|^2+\kappa\f{\varrho^{\gamma}}{\gamma-1} \right) dx$.
\end{theorem}
\begin{remark}
	This result also holds for $p=3$, $q=3$ by small modification in the proof of Theorem \ref{the1.5}.
\end{remark}
\begin{coro}
	Let the weak solution $(\varrho,v)$ of compressible Euler equations meets one of the following conditions
	\begin{enumerate}[(1)]
		\item  $0\leq \varrho \in L^\infty(0,T;L^\infty(\mathbb{T}^d)), \nabla \sqrt{\varrho} \in L^\infty(0,T;L^\infty(\mathbb{T}^d)), \partial_t \sqrt{\varrho}\in  L^\infty(0,T;L^\infty(\mathbb{T}^d))$\\
		$v\in  B^{\beta}_{3,\infty}((0,T);B^{\alpha}_{3,\infty}(\mathbb{T}^{d})) , \beta\geq \alpha>1/3,\ and\ v_0\in L^{2}(\mathbb{T}^d).$
		\item
		$0\leq \varrho \in L^\infty(0,T;L^{\frac{q}{q-3}}(\mathbb{T}^d)),\nabla \sqrt{\varrho}\in L^\infty(0,T;L^{\frac{2q}{q-3}}(\mathbb{T}^d)), \partial_t\sqrt{\varrho}\in  L^\infty(0,T;L^{\frac{2q}{q-3}}(\mathbb{T}^d)),$
		$v\in B^\beta_{3,\infty}(0,T;B^\alpha_{q,\infty}(\mathbb{T}^d)), \beta\geq \alpha >1/3,\ 3< q\leq \frac{3\gamma-1}{\gamma-1}\ and\ v_0\in L^2(\mathbb{T}^d)$.
\item
	$0\leq \varrho \in L^\infty(0,T;L^{\frac{3\gamma-1}{2}}),\nabla \sqrt{\varrho}\in L^\infty(0,T;L^{\frac{2(3\gamma-1)}{2}}(\mathbb{T}^d)), \partial_t\sqrt{\varrho}\in  L^\infty(0,T;L^{\frac{2(3\gamma-1)}{2}}),$\\
	$v\in B^\beta_{3,\infty}(0,T;B^\alpha_{\frac{3\gamma-1}{\gamma-1},\infty}(\mathbb{T}^d)), \beta\geq \alpha >1/3,\ and\ v_0\in L^2(\mathbb{T}^d)$.

	\end{enumerate}
Then the relation   \eqref{energyeq-1} is valid.
\end{coro}
\begin{remark}
In the spirit of above theorems, we will consider the energy equality of weak solutions in the isentropic compressible Navier-Stokes equations without upper bound of the density in a forthcoming paper \cite{[YWY]}.
\end{remark}
We would like to mention that infinitely many weak solutions of the  isentropic compressible Euler equations have been constructed in \cite{[LXX],[CDK],[CVY],[CKMS]}. The  uniqueness of dissipative solutions  with regularity assumption to the
isentropic Euler equations can be found in \cite{[FGJ]}.

The rest of this paper is divided into three sections.
 In Section 2, we present the auxiliary lemmas including  the  Constantin-E-Titi type and Lions type commutators on mollifying kernel.
 Section 3 is devoted to  the Onsager conjecture on the energy conservation for the
isentropic compressible Euler equations.
\section{Notations and some auxiliary lemmas} \label{section2}

First, we introduce some notations used in this paper.
 For $p\in [1,\,\infty]$, the notation $L^{p}(0,\,T;X)$ stands for the set of measurable functions on the interval $(0,\,T)$ with values in $X$ and $\|f(t,\cdot)\|_{X}$ belonging to $L^{p}(0,\,T)$. The classical Sobolev space $W^{k,p}(\Omega)$ is equipped with the norm $\|f\|_{W^{k,p}(\Omega)}=\sum\limits_{|\alpha| =0}^{k}\|D^{\alpha}f\|_{L^{p}(\Omega)}$. For $1\leq q\leq \infty$ and $0<\alpha<1$, the homogeneous Besov space $\dot{B}^{\alpha}_{q,\infty}(\mathbb{T}^{d})$ is the space of functions $f$ on the $d$-dimensional torus $\mathbb{T}^{d}=[0,1]^{d}$ for which the semi-norm
 $$  \|f\|_{\dot{B}^{\alpha}_{q,\infty}(\mathbb{T}^{d})}= \B\|\,|y|^{-\alpha}\B\| f(x-y)-f(x)\B\|_{L_{x}^{q}(\mathbb{T}^{d})}\B\|_{L_{y}^{\infty}(\mathbb{R}^{d})}<\infty,$$
 and the nonhomogeneous Besov space $B^{\alpha}_{q,\infty}(\mathbb{T}^{d})$ is the set of functions $f\in L^{q}(\mathbb{T}^{d})$ for which the norm
 $$ \|f\|_{B^{\alpha}_{q,\infty}(\mathbb{T}^{d})}= \|f\|_{L^{q}(\mathbb{T}^{d})}+\|f\|_{\dot{B}^{\alpha}_{q,\infty}(\mathbb{T}^{d})}<\infty.$$
Likewise, we define  $B^{\alpha}_{q,\infty}((0,T)\times\mathbb{T}^{d})$  via  replacing $\mathbb{T}^{d}$   by $ (0,T)\times\mathbb{T}^{d}$ in the above. A similar definition of Besov norms on the whole space $\mathbb{R}^{d}$ can be referred to in \cite{[Chae]}.

Let $\eta_{\varepsilon}:\mathbb{R}^{d}\rightarrow \mathbb{R}$ be a standard mollifier.i.e. $\eta(x)=C_0e^{-\frac{1}{1-|x|^2}}$ for $|x|<1$ and $\eta(x)=0$ for $|x|\geq 1$, where $C_0$ is a constant such that $\int_{\mathbb{R}^d}\eta (x) dx=1$. For $\varepsilon>0$, we define the rescaled mollifier $\eta_{\varepsilon}(x)=\frac{1}{\varepsilon^d}\eta(\frac{x}{\varepsilon})$. For any function $f\in L^1_{loc}(\Omega)$, its mollified version is defined as
$$f^\varepsilon(x)=(f*\eta_{\varepsilon})(x)=\int_{\mathbb{R}^d}f(x-y)\eta_{\varepsilon}(y)dy,\ \ x\in \Omega_\varepsilon,$$
where $\Omega_\varepsilon=\{x\in \Omega: d(x,\partial\Omega)>\varepsilon\}.$
To deal with the vacuum vase, we need the space-time mollifier.
Abusing notation slightly, we also denote $\eta$ be non-negative   smooth function supported in the space-time ball of radius 1 and its integral equals to 1. We define the rescaled space-time mollifier $\eta_{\varepsilon}(t,x)=\frac{1}{\varepsilon^{d+1}}\eta(\frac{t}{\varepsilon},\frac{x}{\varepsilon})$
\be\label{stm}
f^{\varepsilon}(t,x)=\int_{0}^{T}\int_{\Omega}f(s,y)\eta_{\varepsilon}(t-s,x-y)dyds.
\ee
Moreover, for simplicity, we denote by $$\int_0^T\int_{\mathbb{T}^d} f(t, x)dxdt=\int_0^T\int f\ ~~\text{and}~~ \|f\|_{L^p(0,T;X )}=\|f\|_{L^p(X)}.$$

\begin{definition}\label{wsdefi}
 	A pair ($\varrho,v$) is called a weak solution to \eqref{CEuler} with initial data \eqref{INS1} if ($\varrho,v$) satisfies
 	\begin{enumerate}[(i)]
 		\item equation \eqref{CEuler} holds in $D'(0,T;\Omega)$ and
 		\begin{equation}P(\varrho ), \varrho |v|^2\in L^\infty(0,T;L^1(\Omega)),
 		\end{equation}
 	\item[(ii)]
 	the density $\varrho$ is a renormalized solution of \eqref{CEuler} in the sense of \cite{[PL]}.
 		\item[(iii)]
 		the energy inequality holds
 		\begin{equation}\label{energyineq}
 		\begin{aligned}
 	E(t)   \leq E(0), 		\end{aligned}\end{equation}
 where $E(t)=\int_{\Omega}\left( \frac{1}{2}\varrho |v|^2+\kappa\f{\varrho^{\gamma}}{\gamma-1} \right) dx$.
 	\end{enumerate}
 \end{definition}
The following lemma is the  Lions type commutators on space-time mollifying kernel, which was stated in \cite{[LV],[CY]} and whose proof can be found in \cite{[YWW]}. We refer the reader to \cite{[Lions1]}
for the original version.
\begin{lemma} (\cite{[YWW],[LV],[CY]})
	\label{pLions}Let $1\leq p,q,p_1,q_1,p_2,q_2\leq \infty$,  with $\frac{1}{p}=\frac{1}{p_1}+\frac{1}{p_2}$ and $\frac{1}{q}=\frac{1}{q_1}+\frac{1}{q_2}$.
	Let $\partial$ be a partial derivative in space or time, in addition, let  $\partial_t f,\ \nabla f \in L^{p_1}(0,T;L^{q_1}(\mathbb{T}^d))$, $g\in L^{p_2}(0,T;L^{q_2}(\mathbb{T}^d))$.   Then, there holds \begin{equation}
		\begin{aligned}
	&\|{\partial(fg)^\varepsilon}-\partial(f\,{g}^\varepsilon)\|_{L^p(0,T;L^q(\mathbb{T}^d))}\\
		\leq& C\left(\|\partial_{t} f\|_{L^{p_{1}}(0,T;L^{q_{1}}(\mathbb{T}^d))}+\|\nabla f\|_{L^{p_{1}}(0,T;L^{q_{1}}(\mathbb{T}^d))}\right)\|g\|_{L^{p_{2}}(0,T;L^{q_{2}}(\mathbb{T}^d))},	
		\end{aligned}
	\end{equation}
	for some constant $C>0$ independent of $\varepsilon$, $f$ and $g$. Moreover, $${\partial{(fg)^\varepsilon}}-\partial{(f\,{g^\varepsilon})}\to 0\quad\text{ in } {L^{p}(0,T;L^{q}(\mathbb{T}^d))},$$
	as $\varepsilon\to 0$ if $p_2,q_2<\infty.$
\end{lemma}

 	\begin{lemma}	\label{lem2.3c}
Assume that $0<\alpha_i\leq  \beta_i,\ i=1,2$ and $1\leq p,q,p_{1},q_{1},p_{2},q_{2}\leq\infty$ with $\frac{1}{p}=\frac{1}{p_1}+\frac{1}{p_2}$, $\frac{1}{q}=\frac{1}{q_1}+\frac{1}{q_2}$ . Then, for any $f\in B^{\beta_1}_{p_{1},\infty}((0,T);B^{\alpha_1}_{q_1,\infty}(\mathbb{T}^d)) $, $g\in B^{\beta_2}_{p_{2},\infty}((0,T);B^{\alpha_2}_{q_2,\infty}(\mathbb{T}^d ))$ and $\varepsilon>0$, there holds
	 	\begin{align} \label{gcet}
		\|(fg)^{\varepsilon}- f^{\varepsilon}g^{\varepsilon}\|_{L^p(0,T;L^p(\mathbb{T}^d))} \leq C\varepsilon^{\alpha_1+\alpha_2}\| f\|_{ B^{\beta_1}_{p_{1},\infty}((0,T);B^{\alpha_1}_{q_1,\infty}(\mathbb{T}^d)) }
\|g \|_{B^{\beta_2}_{p_{2},\infty}((0,T);B^{\alpha_2}_{q_2,\infty}(\mathbb{T}^d) )},
\end{align}
where $f^{\varepsilon}$ and $g^{\varepsilon}$ are defined in \eqref{stm}.
\end{lemma}
\begin{remark}
If one modifies the functions $f$ and $g$ just in space direction, the inequality \eqref{gcet}
reduces to
 	\begin{align} \label{cet}
		\|(fg)^{\varepsilon}- f^{\varepsilon}g^{\varepsilon}\|_{L^p(0,T;L^q(\mathbb{T}^d))} \leq C\varepsilon^{\alpha_1+\alpha_2}\| f\|_{L^{p_1}(0,T;\dot{B}^{\alpha_1}_{q_{1},\infty}(\mathbb{T}^d))}
\|g \|_{L^{p_2}(0,T;\dot{B}^{\alpha_2}_{q_{2},\infty}(\mathbb{T}^d))}.	
\end{align}
\end{remark}
\begin{proof}
We recall the following   identity observed  by Constantin-E-Titi   in \cite{[CET]}
\begin{equation}
	\begin{aligned}
&(fg)^{\varepsilon}(t,x)- f^{\varepsilon}g^{\varepsilon}(t,x)\\
=&\int^{\varepsilon}_{-\varepsilon}\int_{\mathbb{T}^{d}}\eta_{\varepsilon}(\tau,y)[f(t-\tau,x-y)-f(t,x)][g(t-\tau, x-y)-g(t,x)]d\tau dy\\&-(f-f^{\varepsilon})(g-g^{\varepsilon})(t,x)\\
=& G(t,x)-(f-f^{\varepsilon})(g-g^{\varepsilon})(t,x).\end{aligned}
\end{equation}
First, the H\"older's inequality and Minkowski inequality yield that
\begin{equation}
	\begin{aligned}
& \|G\|_{L^p((0,T);L^q( \mathbb{T}^d ))}\\ \leq&\Big\|\int^{\varepsilon}_{-\varepsilon}\int_{|y|\leq\varepsilon}\eta_{\varepsilon}(\tau,y)\Big[ f(t-\tau,x-y)-f(t-\tau,x)+f(t-\tau,x)-f(t,x)\Big]\\
&\ \ \ \ \ \ \ \ \ \ \ \Big[ g(t-\tau,x-y)-g(t-\tau,x)+g(t-\tau,x)-g(t,x)\Big]dy d\tau\Big\|_{L^p(L^q)}\\
\leq& \int^{\varepsilon}_{-\varepsilon}\int_{|y|\leq\varepsilon}\eta_{\varepsilon}(\tau,y)\Big[\|f(t-\tau,x-y)-f(t-\tau,x)\|_{L^{p_1}(L^{q_1})}+\|f(t-\tau,x)-f(t,x)\|_{L^{p_1}(L^{q_1})}\Big]\\
&\ \ \ \ \ \ \ \ \ \ \  \Big[\|g(t-\tau,x-y)-g(t-\tau,x)\|_{L^{p_2}(L^{q_2})}+\|g(t-\tau,x)-g(t,x)\|_{L^{p_2}(L^{q_2})}\Big]dyd\tau\\
\leq & \int^{\varepsilon}_{-\varepsilon}\int_{|y|\leq\varepsilon}\eta_{\varepsilon}(\tau,y)\Big[\varepsilon^{\alpha_1}\|f(t-\tau)\|_{B^{\beta_1}_{p_1,\infty}(B^{\alpha_1}_{q_1,\infty})}+\varepsilon^{\beta_1}\|f(t)\|_{B^{\beta_1}_{p_1,\infty}(L^{q_1})}\Big]\\
&\ \ \ \ \ \ \ \ \ \ \ \Big[\varepsilon^{\alpha_2}\|f(t-\tau)\|_{B^{\beta_2}_{p_2,\infty}(B^{\alpha_2}_{q_2,\infty})}+\varepsilon^{\beta_2}\|f(t)\|_{B^{\beta_2}_{p_2,\infty}(L^{q_2})}\Big]dyd\tau\\
\leq&  C\varepsilon^{\alpha_1+\alpha_2}\| f\|_{ B^{\beta_1}_{p_{1},\infty}((0,T);B^{\alpha_1}_{q_1,\infty}(\mathbb{T}^d)) }
\|g \|_{B^{\beta_2}_{p_{2},\infty}((0,T);B^{\alpha_2}_{q_2,\infty}(\mathbb{T}^d) )},
\end{aligned}\end{equation}
and
\begin{equation}
	\begin{aligned}
		&\|(f-f^{\varepsilon})(g-g^{\varepsilon})\|_{L^p(0,T;L^q(\mathbb{T}^d))}\\
		\leq& C\|f-f^\varepsilon\|_{L^{p_1}(L^{q_1})}\|g-g^\varepsilon\|_{L^{p_2}(L^{q_2})}\\
		\leq& C\varepsilon^{\alpha_1+\alpha_2}\| f\|_{ B^{\beta_1}_{p_{1},\infty}((0,T);B^{\alpha_1}_{q_1,\infty}(\mathbb{T}^d)) }
		\|g \|_{B^{\beta_2}_{p_{2},\infty}((0,T);B^{\alpha_2}_{q_2,\infty}(\mathbb{T}^d) )},
	\end{aligned}
\end{equation}
where we have used the following facts that for any $u\in{B}^{\beta} _{p,\infty}((0,T);B^\alpha_{q,\infty}( \mathbb{T}^d))$ with $\beta\geq \alpha>0$, $1\leq p,q\leq\infty$ and a.e.$\,y\in \mathbb{R}^{d}$,
$$\ba
&\| u(\cdot-y)-u(\cdot)\|_{ L^{p}(0,T;L^q(\mathbb{T}^d))}\leq C|y|^{\alpha}\|u\|_{{B}^{\beta} _{p,\infty}((0,T);B^\alpha_{q,\infty}( \mathbb{T}^d))},\\
&\| u ^\varepsilon-u \|_{ L^{p}((0,T);L^q( \mathbb{T}^d))}\leq C\varepsilon^{\alpha}\|u\|_{{B}^{\beta} _{p,\infty}((0,T);B^\alpha_{q,\infty}( \mathbb{T}^d))},\\
&\|\nabla u^\varepsilon\|_{L^p(0,T;L^q(\mathbb{T}^d))}\leq C\varepsilon^{\alpha-1}\|u\|_{{B}^{\beta} _{p,\infty}((0,T);B^\alpha_{q,\infty}( \mathbb{T}^d))},
\ea$$
which can be deduced from periodicity of the function $u$ in essentially the same manner as derivation of \cite[Lemma 2.1]{[Chae]} and \cite{[EGSW]}.
The proof of this lemma is completed.
\end{proof}
 \begin{lemma}\label{lem2.2}
	Let $ p,q,p_1,q_1,p_2,q_2\in[1,+\infty)$ with $\frac{1}{p}=\frac{1}{p_1}+\frac{1}{p_2},\frac{1}{q}=\frac{1}{q_1}+\frac{1}{q_2} $. Assume $f\in L^{p_1}(0,T;L^{q_1}(\mathbb{T}^d) $ and $g\in L^{p_2}(0,T;L^{q_2}(\mathbb{T}^d))$. Then for any $\varepsilon>0$, there holds
	\begin{equation}\label{a4}
	\|(fg)^\varepsilon-f^\varepsilon g^\varepsilon\|_{L^p(0,T;L^q(\mathbb{T}^d))}\rightarrow 0,\ \ \ as\ \varepsilon\rightarrow 0.
	\end{equation}
\end{lemma}
\begin{proof}
	By the triangle inequality, one have
	\begin{equation*}
	\begin{aligned}
	&\|(fg)^\varepsilon-f^\varepsilon g^\varepsilon\|_{L^p(0,T;L^q(\mathbb{T}^d))}\\
	\leq & C\left(\|(fg)^\varepsilon- (fg)\|_{L^p(0,T;L^q(\mathbb{T}^d))}+\|fg-f^\varepsilon g\|_{L^p(0,T;L^q(\mathbb{T}^d))}+\|f^\varepsilon g-f^\varepsilon g^\varepsilon\|_{L^p(0,T;L^q(\mathbb{T}^d))}\right)\\
	\leq &C\Big(\|(fg)^\varepsilon- fg\|_{L^p(0,T;L^q(\mathbb{T}^d))}+\|f-f^\varepsilon\|_{L^{p_1}(0,T;L^{q_1}(\mathbb{T}^d))}\|g\|_{L^{p_2}(0,T;L^{q_2}(\mathbb{T}^d))}\\
	&\ \ \ \ \ +\|f^\varepsilon\|_{L^{p_1}(0,T;L^{q_1}(\mathbb{T}^d))}\|g-g^\varepsilon\|_{L^{p_2}(0,T;L^{q_2}(\mathbb{T}^d))}\Big),
	\end{aligned}
	\end{equation*}
	then, together with the properties of the standard mollification, we can obtain \eqref{a4}.
\end{proof}

\begin{lemma}[\cite{[Simon]}]\label{AL}
	Let $X\hookrightarrow B\hookrightarrow Y$ be three Banach spaces with compact imbedding $X \hookrightarrow\hookrightarrow Y$. Further, let there exist $0<\theta <1$ and $M>0$ such that
	\begin{equation}\label{le1}
	\|v\|_{B}\leq M\|v\|_{X}^{1-\theta}\|v\|_{Y}^\theta\ \ for\ all\ v\in X\cap Y.\end{equation}
Denote for $T>0$,
\begin{equation}\label{le2}
	W(0,T):=W^{s_0,r_0}((0,T), X)\cap W^{s_1,r_1}((0,T),Y)
\end{equation}
with
\begin{equation}\label{le3}
	\begin{aligned}
		&s_0,s_1 \in \mathbb{R}; \ r_0, r_1\in [1,\infty],\\
		s_\theta :=(1-\theta)s_0&+\theta s_1,\ \f{1}{r_\theta}:=\f{1-\theta}{r_0}+\f{\theta}{r_1},\ s^{*}:=s_\theta -\f{1}{r_\theta}.
	\end{aligned}
\end{equation}
Assume that $s_\theta>0$ and $F$ is a bounded set in $W(0,T)$. Then, we have

If $s_{*}\leq 0$, then $F$ is relatively compact in $L^p((0,T),B)$ for all $1\leq p< p^{*}:=-\f{1}{s^{*}}$.

If $s_{*}> 0$, then $F$ is relatively compact in $C((0,T),B)$.

\end{lemma}
\section{Onsager conjecture on the energy conservation for the
isentropic compressible Euler equations}
 We will use the framework of \cite{[CY]} to prove main theorems in this section.
 First, we reformulate the equations after mollifier the equations. Second, we invoke the  Constantin-E-Titi type and Lions type commutators on mollifying kernel to pass the limits.
 Third,  it is to  get the energy  conservation  up to the initial time.

\subsection {Non-vacuum case}
  \begin{proof}[Proof of Theorem \ref{the1.1}]

For the non-vacuum case,  it is  sufficient to  mollify $v$ in   space direction. For any smooth function $\phi(t)$ which is compact support in $(0,+\infty)$, multiplying the momentum equation in  $\eqref{CEuler}$ by $(  \phi(t)v^{\varepsilon})^\varepsilon$, then integrating over $(0,T)\times \Omega$ , we  arrive at
\begin{equation}\label{ec1} \begin{aligned}
		\int_0^T\int \phi(t) v^{\varepsilon}\B[\partial_{t}(\varrho v)^{\varepsilon}+ \Div(\varrho v\otimes v)^{\varepsilon}+\kappa\nabla \pi(\varrho)^\varepsilon \B]=0.
\end{aligned}\end{equation}
To apply the  commutators on mollifying kernel to pass the limit of $\varepsilon$,
we first need to rewrite equations \eqref{ec1}.
By some  straightforward  computation, we see that
\begin{equation}\label{ec2}
	\begin{aligned}
		\int_0^T\int \phi(t) v^{\varepsilon} \partial_{t} (\varrho v )^{\varepsilon}=&\int_0^T\int \phi(t) v^{\varepsilon}\B[ \partial_{t} (\varrho v )^{\varepsilon}-\partial_{t}(\varrho v^{\varepsilon})\B]+ \int_0^T\int \phi(t) v^{\varepsilon} \partial_{t}(\varrho v^{\varepsilon}) \\
		=& \int_0^T\int \phi(t) v^{\varepsilon} \B[\partial_{t} (\varrho v )^{\varepsilon}-\partial_{t}(\varrho v^{\varepsilon})\B]+\int_0^T\int \phi(t) \varrho\partial_t{\frac{|v^{\varepsilon}|^2}{2}} \\
		&+\int_0^T\int \phi(t) \varrho_t|v^{\varepsilon}|^2.
\end{aligned}\end{equation}
Using integration by parts  many times and the equation  $\eqref{CEuler}_1$, one deduces that
\begin{align}\label{ec3}
	&\int_0^T\int \phi(t) v^{\varepsilon} \Div(\varrho v\otimes v)^{\varepsilon}\nonumber\\
	=& \int_0^T\int  \phi(t) v^{\varepsilon}  \Div[(\varrho v\otimes v)^{\varepsilon}-(\varrho  v)\otimes v^{\varepsilon}]+\int_0^T\int \phi(t) v^{\varepsilon}\Div(\varrho  v\otimes v^{\varepsilon})\nonumber\\
	=& -\int_0^T\int \phi(t) \nabla v^{\varepsilon}  [(\varrho v\otimes v)^{\varepsilon}-\varrho((v\otimes v)^{\varepsilon})+\varrho((v\otimes v)^{\varepsilon})-(\varrho  v)\otimes v^{\varepsilon}]\nonumber\\&+  \int_0^T\int \phi(t) \left(\Div (\varrho v ) |v^{\varepsilon}|^{2}+\f12 \varrho v \nabla|v^{\varepsilon}  |^{2}
	\right)\nonumber\\
	=&  \int_0^T\int \phi(t)  v^{\varepsilon} \text{div} [(\varrho v\otimes v)^{\varepsilon}-\varrho((v\otimes v)^{\varepsilon})]-\int_0^T\int \phi(t)  \nabla v^{\varepsilon}\varrho[ (v\otimes v)^{\varepsilon} -v\otimes v^{\varepsilon}]\nonumber\\&+\f{1}{2}\int_0^T\int \phi(t) \Div (\varrho v ) |v^{\varepsilon}|^{2}\nonumber\\
	=&  \int_0^T\int \phi(t)  v^{\varepsilon} \text{div} [(\varrho v\otimes v)^{\varepsilon}-\varrho(v\otimes v)^{\varepsilon}]-\int_0^T\int \phi(t)  \nabla v^{\varepsilon}\varrho[ (v\otimes v)^{\varepsilon} -    v^{\varepsilon}\otimes v^{\varepsilon}]\nonumber\\&-\int_0^T\int \phi(t)  \nabla v^{\varepsilon}\varrho[ (  v^{\varepsilon}\otimes v^{\varepsilon}) -    v^{\varepsilon}\otimes v^{\varepsilon}]-\frac{1}{2}\int_0^T\int \phi(t) \partial_t \varrho |v^{\varepsilon}|^{2}.
\end{align}
Thanks to  integration by parts once again, we infer that
\begin{align}
&\int_0^T\int \phi(t) \nabla v^{\varepsilon}\varrho[ (  v^{\varepsilon}\otimes v^{\varepsilon}) -    v\otimes v^{\varepsilon}] \nonumber\\
=&\f12\int_0^T\int  \phi(t) \nabla |v^{\varepsilon}|^{2}\varrho(v^{\varepsilon}-v)\nonumber\\
=&-\f12\int_0^T\int  \phi(t) |v^{\varepsilon}|^{2}\text{div}[\varrho v^{\varepsilon}-\varrho v]\nonumber\\
=&-\f12\int_0^T\int  \phi(t) |v^{\varepsilon}|^{2}\text{div}[\varrho v^{\varepsilon}-(\varrho v)^{\varepsilon}+(\varrho v)^{\varepsilon}-\varrho v]
\nonumber\\
=&-\f12\int_0^T\int \phi(t)   |v^{\varepsilon}|^{2}\text{div}[\varrho v^{\varepsilon}-(\varrho v)^{\varepsilon}]+\f12\int_0^T\int  \phi(t) |v^{\varepsilon}|^{2}(\varrho^{\varepsilon}_{t}-\varrho_{t}).\label{e3.4}
\end{align}
Inserting  \eqref{e3.4} into \eqref{ec3}, we arrive at
\begin{align}
	&\int_0^T\int \phi(t)v^{\varepsilon} \Div(\varrho v\otimes v)^{\varepsilon}\nonumber\\
	=&  \int_0^T\int \phi(t)  v^{\varepsilon} \text{div} [(\varrho v\otimes v)^{\varepsilon}-\varrho(v\otimes v)^{\varepsilon}]-\int_0^T\int \phi(t) \nabla v^{\varepsilon}\varrho[ (v\otimes v)^{\varepsilon} -    v^{\varepsilon}\otimes v^{\varepsilon}]\nonumber\\& +\f12\int_0^T\int  \phi(t)  |v^{\varepsilon}|^{2}\text{div}[\varrho v^{\varepsilon}-(\varrho v)^{\varepsilon}]-\f12\int_0^T\int  \phi(t)  |v^{\varepsilon}|^{2}(\varrho^{\varepsilon}_{t}-\varrho_{t})-\frac{1}{2}\int_0^T\int \phi(t)\partial_t \varrho |v^{\varepsilon}|^{2}. \label{ec31}
\end{align}
We  also reformulate the pressure term as
\begin{equation}\label{ec6}
	\begin{aligned}
		&\kappa\int_0^T\int \phi(t) v^{\varepsilon}\nabla (\varrho^\gamma)^{\varepsilon}= \kappa\int_0^T\int \phi(t) [v^{\varepsilon}\nabla(\varrho^\gamma)^{\varepsilon}-v\nabla (\varrho^\gamma)]+\kappa\int_0^T\int \phi(t) v \nabla (\varrho^\gamma).
\end{aligned} \end{equation}
With the help of  the integration by parts and the mass equation in $\eqref{CEuler}$   again, we
know that
$$\ba
\kappa\int_0^T\int \phi(t) v \cdot\nabla (\varrho^\gamma)
=&-\kappa\int_0^T\int \phi(t) \varrho^{\gamma-1} \varrho\text{div\,}v\\
=&\kappa\int_0^T\int \phi(t) \varrho^{\gamma-1}(\partial_{t}\varrho+v\cdot\nabla\varrho)
\\
=&\kappa\f{1}{\gamma}\int_0^T\int \phi(t)  \partial_{t}\varrho^{\gamma } +\kappa\f{1}{\gamma}\int_0^T\int \phi(t) v\cdot\nabla\varrho^{\gamma },
\ea $$
which leads to  that
\begin{equation}\label{ec62}\begin{aligned}
		\kappa\int_0^T\int \phi(t) v \cdot\nabla (\varrho^\gamma)
		= \kappa\f{1}{\gamma-1}\int_0^T\int\phi(t) \partial_{t}\varrho^{\gamma }.
\end{aligned}\end{equation}
Plugging \eqref{ec62} into \eqref{ec6}, we conclude that
\begin{equation}\label{ec3.8}
	\begin{aligned}
		&\kappa\int_0^T\int \phi(t) v^{\varepsilon}\nabla (\varrho^\gamma)^{\varepsilon}= \kappa\int_0^T\int \phi(t) [v^{\varepsilon}\nabla(\varrho^\gamma)^{\varepsilon}-v\nabla (\varrho^\gamma)]+ \kappa\f{1}{\gamma-1}\int_0^T\int \phi(t) \partial_{t}\varrho^{\gamma }.
\end{aligned} \end{equation}
Substituting  \eqref{ec2},   \eqref{e3.4}  and     \eqref{ec3.8} into \eqref{ec1}, we observe that
	\begin{align} \label{1ec9}
		&-\int_0^T\int \phi(t) \partial_t\left(\varrho{\frac{|v^{\varepsilon}|^2}{2}}+\kappa\f{1}{\gamma-1}\varrho^{\gamma }\right)\nonumber \\
		=&-\int_0^T\int \phi(t)  v^{\varepsilon} \B[\partial_{t} (\varrho v )^{\varepsilon}-\partial_{t}(\varrho v^{\varepsilon})\B]\nonumber\\&+\int_0^T\int   \phi(t) v^{\varepsilon} \text{div} [(\varrho v\otimes v)^{\varepsilon}-\varrho(v\otimes v)^{\varepsilon}]-\int_0^T\int \phi(t) \nabla v^{\varepsilon}\varrho[ (v\otimes v)^{\varepsilon} -    v^{\varepsilon}\otimes v^{\varepsilon}]\nonumber\\ &+\f12\int_0^T\int \phi(t)   |v^{\varepsilon}|^{2}\text{div}[\varrho v^{\varepsilon}-(\varrho v)^{\varepsilon}]-\f12\int_0^T\int  \phi(t) |v^{\varepsilon}|^{2}(\varrho^{\varepsilon}_{t}-\varrho_{t})\nonumber\\&-\int_0^T\int \phi(t) \kappa[v^{\varepsilon}\nabla(\varrho^\gamma)^{\varepsilon}-v\nabla (\varrho^\gamma)].
\end{align}

At this stage, we can apply the commutators on mollifying kernel to pass the limit
in the right hand-side of \eqref{ec9}.
In light of the H\"older's  inequality and Lions type commutators on mollifying kernel Lemma \eqref{pLions}, we get
\begin{equation}\label{e3.81}\begin{aligned}
\int_s^t\int \phi(t) v^{\varepsilon} \B[\partial_{t} (\varrho v )^{\varepsilon}-\partial_{t}(\varrho v^{\varepsilon})\B]&\leq C\|v^{\varepsilon}\|_{L^{p}(L^{q})}
\|\partial_{t} (\varrho v )^{\varepsilon}-\partial_{t}(\varrho v^{\varepsilon})\|_{L^{\f{p}{p-1}}(L^{\f{q}{q-1}})}\\
&\leq C\|v\|^{2}_{L^{p}(L^{q})}
 (\|\varrho_{t}\|_{L^{\f{p}{p-2}}(L^{\f{q}{q-2}})}+\|\nabla \varrho\|_{L^{\f{p}{p-2}}(L^{\f{q}{q-2}})} );
\end{aligned}\end{equation}
 \begin{equation}\label{e3.11}\begin{aligned}
 		  &|\int_0^T\int \phi(t)  v^{\varepsilon} \text{div} [(\varrho v\otimes v)^{\varepsilon}-\varrho(v\otimes v)^{\varepsilon}]|\\
 		\leq & C \|\text{div} [(\varrho v\otimes v)^{\varepsilon}-\varrho(v\otimes v)^{\varepsilon}]\|_{L^{\f{p}{p-1}}(L^{\f{q}{q-1}})}\| v^{\varepsilon}\|_{L^{p}(L^{q})}\\
\leq &  C\|\text{div} [(\varrho v\otimes v)^{\varepsilon}-\varrho(v\otimes v)^{\varepsilon}]\|_{L^{\f{p}{p-1}}(L^{\f{q}{q-1}})}\| v^{\varepsilon}\|_{L^{p}(L^{q})}\\
\leq & C\| v^{\varepsilon}\|^{3}_{L^{p}(L^{q})}
 (\|\varrho_{t}\|_{L^{\f{p}{p-3}}(L^{\f{q}{q-3}})}+\|\nabla \varrho\|_{L^{\f{p}{p-3}}(L^{\f{q}{q-3}})} );
 \end{aligned}\end{equation}
 and
  \be\ba\label{e3.83}
|\int_0^T\int    |v^{\varepsilon}|^{2}\text{div}[\varrho v^{\varepsilon}-(\varrho v)^{\varepsilon}]|\leq& C\| |v^{\varepsilon}|^{2}\|_{L^{\frac{p}{2}}(L^\frac{q}{2})}\|\text{div}[\varrho v^{\varepsilon}-(\varrho v)^{\varepsilon}]\|_{L^{\frac{p}{p-2}}(L^\frac{q}{q-2})}\\
 \leq &C\| |v \|^{3}_{L^{\frac{p}{2}}(L^\frac{q}{2})}
 (\|\varrho_{t}\|_{L^{\f{p}{p-3}}(L^{\f{q}{q-3}})}+\|\nabla \varrho\|_{L^{\f{p}{p-3}}(L^{\f{q}{q-3}})} ).
  \ea\ee
In addition, we derive from
  Lemma \ref{pLions} that, as $\varepsilon\rightarrow0$,
 $$\ba
 \int_0^T\int  v^{\varepsilon} \B[\partial_{t} (\varrho v )^{\varepsilon}-\partial_{t}(\varrho v^{\varepsilon})\B]\rightarrow0,\\
 \int_0^T\int    |v^{\varepsilon}|^{2}\text{div}[\varrho v^{\varepsilon}-(\varrho v)^{\varepsilon}]\rightarrow0,\\
 \int_0^T\int    |v^{\varepsilon}|^{2}\text{div}[\varrho v^{\varepsilon}-(\varrho v)^{\varepsilon}]\rightarrow0.\ea$$
The classical  Constantin-E-Titi type
  commutators on mollifying kernel   \eqref{cet}
 and
 the H\"older's inequality    allow  us to obtain
\be\ba\label{ce3.13}
 &\int_0^T\int \phi(t)  \nabla v^{\varepsilon}\varrho[ (v\otimes v)^{\varepsilon} -    v^{\varepsilon}\otimes v^{\varepsilon}] \\
 \leq& C\|\varrho\|_{L^{ \f{p}{p-3}}(L^{ \f{q}{q-3}})}\| (v\otimes v)^{\varepsilon} -    v^{\varepsilon}\otimes v^{\varepsilon}\|_{L^{\frac{p}{2}}(L^\frac{q}{2})}\|\nabla v^{\varepsilon}\|_{L^{p}(L^{q})}\\
 \leq& C\varepsilon^{3\alpha-1}\|\varrho\|_{L^{ \f{p}{p-3}}(L^{ \f{q}{q-3}})}\|  v^{\varepsilon}\|^{3}_{L^{p}(B^{\alpha}_{q,\infty}) }.
 \end{aligned}\end{equation}
 which in turn means
 \be\label{c3.16}
 \int_0^T\int\phi(t)  \nabla v^{\varepsilon}\varrho[ (v\otimes v)^{\varepsilon} -    v^{\varepsilon}\otimes v^{\varepsilon}]\rightarrow0.
 \ee
 The H\"older's inequality guarantees that
 \be\ba
\int_0^T\int  \phi(t) |v^{\varepsilon}|^{2}(\varrho^{\varepsilon}_{t}-\varrho_{t})
 \leq&C \| |v^{\varepsilon}|^{2}\|_{L^{\frac{p}{2}}(L^\frac{q}{2})}\|\varrho^{\varepsilon}_{t}-\varrho_{t}\|_{L^{\frac{p}{p-2}}(L^\frac{q}{q-2})}.
   \ea\ee
As a consequence,   the standard properties of the mollification help us to see that,
as $\varepsilon\rightarrow0$,
 $$\int_0^T\int   \phi(t) |v^{\varepsilon}|^{2}(\varrho^{\varepsilon}_{t}-\varrho_{t})\rightarrow0.$$
 It follows from the H\"older's inequality  that
\be\ba\label{pterm}
 \|\nabla(\varrho^\gamma)\|_{L^{\f{p}{p-1}}(L^{\f{q}{q-1}})}=&\gamma\|\nabla\varrho(\varrho^{\gamma-1})\|_{L^{\f{p}{p-1}}(L^{\f{q}{q-1}})}\\
 \leq&
 C \|\nabla \varrho\|_{L^{\f{p}{p-3}}(L^{\f{q}{q-3}})}
 \|\varrho \|^{\gamma-1}_{L^{\f{p(\gamma-1)}{2}}L^{\f{q(\gamma-1)}{2}}},
 \ea\ee
 which in turn implies that
 \begin{equation}\label{ec10}\begin{aligned}
 \int_0^T\int \phi(t) [v^{\varepsilon}\nabla(\varrho^\gamma)^{\varepsilon}-v\nabla (\varrho^\gamma)]\rightarrow0.
 \end{aligned}\end{equation}
Combining the above estimates  \eqref{e3.81}-\eqref{c3.16} and \eqref{ec9},   we get
\begin{equation}\label{c11}
	\begin{aligned}
		 \int_0^T    \int_{\mathbb{T}^d}\phi(t)\partial_{t}\left( \frac{1}{2}\varrho |v|^2+\kappa\f{\varrho^{\gamma}}{\gamma-1} \right) dx =0.	
	\end{aligned}
\end{equation}
\end{proof}

  \begin{proof}[Proof of Theorem \ref{the1.3}]
A slight modify the above proof allows us to obtain the global energy conservation up to the initial time. Indeed, repeating the derivation above, we infer that
  	\begin{align} \label{ec9}
		&-\int_0^T\int \phi(t)_t\left(\varrho{\frac{|v^{\varepsilon}|^2}{2}}+\kappa\f{1}{\gamma-1}\varrho^{\gamma }\right)\nonumber \\
		=&-\int_0^T\int \phi(t)v^{\varepsilon} \B[\partial_{t} (\varrho v )^{\varepsilon}-\partial_{t}(\varrho v^{\varepsilon})\B]\nonumber\\&+\int_0^T\int\phi(t)   v^{\varepsilon} \text{div} [(\varrho v\otimes v)^{\varepsilon}-\varrho(v\otimes v)^{\varepsilon}]-\int_0^T\int\phi(t)  \nabla v^{\varepsilon}\varrho[ (v\otimes v)^{\varepsilon} -    v^{\varepsilon}\otimes v^{\varepsilon}]\nonumber\\ &+\f12\int_0^T\int\phi(t)   |v^{\varepsilon}|^{2}\text{div}[\varrho v^{\varepsilon}-(\varrho v)^{\varepsilon}]-\f12\int_0^T\int\phi(t)   |v^{\varepsilon}|^{2}(\varrho^{\varepsilon}_{t}-\varrho_{t})\nonumber\\&-\int_0^T\int\phi(t) \kappa[v^{\varepsilon}\nabla(\varrho^\gamma)^{\varepsilon}-v\nabla (\varrho^\gamma)].
\end{align}
Note that  the triangle inequality and H\"older's inequality  guarantee that
  \begin{equation}\label{3.91}\begin{aligned}
  		\|\partial_t \varrho\|_{L^{\f{p}{p-3}}(L^{\f{q}{q-3}})}\leq & C\|\sqrt{\varrho}\partial_t \sqrt{\varrho}\|_{L^{\f{p}{p-3}}(L^{\f{q}{q-3}})}\leq C\|\sqrt{\varrho}\|_{L^{2k}(L^{2l})}\|\partial_t \sqrt{\varrho}\|_{L^{\f{2kp}{2k(p-3)-p}}(L^{\f{2lq}{2l(q-3)-q}})},
  \end{aligned}\end{equation}
  and
  \begin{equation}\label{3.92}
  	\begin{aligned}
  		\|\nabla \varrho\|_{L^{\f{p}{p-3}}(L^{\f{q}{q-3}})}\leq & C\|\sqrt{\varrho}\nabla \sqrt{\varrho}\|_{L^{\f{p}{p-3}}(L^{\f{q}{q-3}})}\leq C\|\sqrt{\varrho}\|_{L^{2k}(L^{2l})}\|\nabla \sqrt{\varrho}\|_{L^{\f{2kp}{2k(p-3)-p}}(L^{\f{2lq}{2l(q-3)-q}})}.
  	\end{aligned}
  \end{equation}
Substituting \eqref{3.91} and \eqref{3.91} into \eqref{e3.81}-\eqref{e3.83} and \eqref{pterm}, following the path of \eqref{c11}, we conclude that
$$
		 -\int_0^T\int \phi(t)_t\left(\varrho{\frac{|v^{\varepsilon}|^2}{2}}+\kappa\f{1}{\gamma-1}\varrho^{\gamma }\right)=0.	$$
The next objective is to get the energy equality up to the initial time $t=0$ by the similar method in \cite{[Yu2]}, for the convenience of the reader and the integrity of the paper, we give the details.

First we prove the continuity of $\sqrt{\varrho}v(t)$ in the strong topology as $t\to 0^+$.  To do this, we define the function $f$ on $[0,T]$ as
$$f(t)=\int_{\mathbb{T}^d}(\varrho v)(t,x)\cdot \varphi(x) dx,\ for\ any\ \varphi(x)\in \mathfrak{D}(\mathbb{T}^d),$$
which is a continuous function with respect to $t\in [0,T]$. Moreover, since
$$\varrho \in L^\infty{0,T; L^\gamma (\mathbb{T}^d)}\ and \ \sqrt{\varrho}v\in L^\infty(0,T;L^2(\mathbb{T}^d)),$$
we can obtain $\varrho v\in L^\infty(0,T;L^{\frac{2\gamma }{\gamma+1}}(\mathbb{T}^d)).$\\
From the moument equation, we have
$$\frac{d}{dt}\int_{\mathbb{T}^n} (\varrho v)(t,x)\cdot \varphi(x) dx=\int_{\mathbb{T}^n}\varrho v\otimes v:\nabla \varphi(x)-\pi\Div\varphi(x)dx,$$
which is bounded for any function $\varphi\in \mathfrak{D}(\mathbb{T}^d)$. Then it follows from the Corollary 2.1 in  \cite{[Feireisl2004]} that
\begin{equation}\label{c14}
	\varrho v\in C([0,T];L^{\frac{2\gamma}{\gamma +1}}_{\text{weak}}(\mathbb{T}^d)).
\end{equation}
On the other hand, since
$$\ba
&\nabla \varrho^\gamma\in L^{\f{kp}{k(p-3)+(\gamma-1)p}}( L^{\f{lq}{l(q-3)+(\gamma-1)q}}),\\&\ \partial_t\varrho^\gamma\in L^{\f{kp}{k(p-3)+(\gamma-1)p}}(L^{\f{lq}{l(q-3)+(\gamma-1)q}})\hookrightarrow L^{\f{kp}{k(p-2)+(\gamma-1)p}}(L^{\f{lq}{l(q-2)+(\gamma-1)q}}),\ea$$
and
$$\nabla \sqrt{\varrho}\in L^{\f{2kp}{2k(p-3)-p}}( L^{\f{2lq}{2l(q-3)-q}}),\ \partial_t\sqrt{\varrho}\in L^{\f{2kp}{2k(p-3)-p}}(L^{\f{2lq}{2l(q-3)-q}})\hookrightarrow L^{\f{2kp}{2k(p-3)-p}}(L^{\f{2lq}{2l(q-2)-q}}).$$
Hence, using the Aubin-Lions Lemma \ref{AL}, we can obtain
\begin{equation}\label{c15}
	\varrho^\gamma\in C([0,T];L^{\frac{lq}{l(q-3)+(\gamma-1)q}}(\mathbb{T}^d))  \,and\,\ \sqrt{\varrho }\in C([0,T];L^{\f{2lq}{2l(q-3)-q}}(\mathbb{T}^d)),\
\end{equation}
for $k\geq \frac{(\gamma-1)(d+q)p}{2q-d(p-3)}$, $p>3$ and $q>\max\{3, \frac{d(p-3)}{2}\}$.

Meanwhile, using the natural energy \eqref{energyineq}, \eqref{c14} and \eqref{c15}, we have
\begin{equation}\label{c16}
	\begin{aligned}
		0&\leq \overline{\lim_{t\rightarrow 0}}\int |\sqrt{\varrho} v-\sqrt{\varrho_0}v_0|^2 dx\\
		&=2\overline{\lim_{t\rightarrow 0}}\left(\int \left(\f{1}{2}\varrho |v|^2 +\f{1}{\gamma -1}\varrho ^\gamma \right)dx-\int\left(\f{1}{2}\varrho_0 |v_0|^2+\f{1}{\gamma -1}\varrho_0 ^\gamma \right)dx\right)\\
		&\ \ \ +2\overline{\lim_{t\rightarrow 0}}\left(\int\sqrt{\varrho_0}v_0\left(\sqrt{\varrho_0}v_0-\sqrt{\varrho} v\right)dx+\f{1}{\gamma -1}\int \left(\varrho_0^\gamma -\varrho^\gamma\right)dx\right)\\
		&\leq 2\overline{\lim_{t\rightarrow 0}}\int \sqrt{\varrho_0}v_0\left(\sqrt{\varrho_0}v_0-\sqrt{\varrho}v\right)dx\\
		&=2\overline{\lim_{t\rightarrow 0}}\int v_0 \left(\varrho_0 v_0 -\varrho v\right)dx+2\overline{\lim_{t\rightarrow 0}}\int v_0 \sqrt{\varrho }v\left(\sqrt{\varrho }-\sqrt{\varrho_0}\right)dx=0,
	\end{aligned}
\end{equation}
from
which it follows
\begin{equation}\label{c17}
	\sqrt{\varrho} v(t)\rightarrow \sqrt{\varrho }v(0)\ \ strongly\ in\ L^2(\mathbb{T}^d)\ as\ t\rightarrow 0^+.
\end{equation}
Similarly, one has the right temporal continuity  of $\sqrt{\varrho}v$ in $L^2(\mathbb{T}^d)$, hence, for any $t_0\geq 0$, we infer that
\begin{equation}\label{c18}
	\sqrt{\varrho} v(t)\rightarrow \sqrt{\varrho }v(t_0)\ \ strongly\ in\ L^2(\mathbb{T}^d)\ as\ t\rightarrow t_0^+.
\end{equation}
Before we go any further, it should be noted that \eqref{c11} remains valid for function $\phi(t)$ belonging to $W^{1,\infty}$ rather than $C^1$, then for any $t_0>0$, we redefine the test function $\phi(t)$ as $\phi_\tau$ for some positive $\tau$ and $\alpha $ such that $\tau +\alpha <t_0$, that is
\begin{equation}
	\phi_\tau(t)=\left\{\begin{array}{lll}
		0, & 0\leq t\leq \tau,\\
		\f{t-\tau}{\alpha}, & \tau\leq t\leq \tau+\alpha,\\
		1, &\tau+\alpha \leq t\leq t_0,\\
		\f{t_0-t}{\alpha }, & t_0\leq t\leq t_0 +\alpha ,\\
		0, & t_0+\alpha \leq t.
	\end{array}\right.
\end{equation}
Then substituting this test function into \eqref{c11}, we arrive at
\begin{equation}
	\begin{aligned}
		-\int_\tau^{\tau+\alpha}\int& \f{1}{\alpha}\left(\f{1}{2}\varrho v^2+\kappa\f{1}{\gamma-1}\varrho^\gamma \right)+\f{1}{\alpha}\int_{t_0}^{t_0+\alpha}\int 	\left(\f{1}{2}\varrho v^2+\kappa\f{1}{\gamma-1}\varrho^\gamma \right)=0.
	\end{aligned}	
\end{equation}
Taking $\alpha\rightarrow 0$  and the  Lebesgue point Theorem, we deduce that
\begin{equation}
	\begin{aligned}
		-\int&\left(\f{1}{2}\varrho v^2+\kappa\f{1}{\gamma-1}\varrho^\gamma \right)(\tau)dx+\int\left(\f{1}{2}\varrho v^2+\kappa\f{1}{\gamma-1}\varrho^\gamma \right)(t_0)dx=0.
	\end{aligned}
\end{equation}
Finally, letting $\tau\rightarrow 0$, using \eqref{c14} and \eqref{c17}, we can obtain
\begin{equation}\ba
	\int\left(\f{1}{2}\varrho v^2+\kappa\f{1}{\gamma-1}\varrho^\gamma \right)(t_0)dx=&\int\left(\f{1}{2}\varrho_0 v_0^2+\kappa\f{1}{\gamma-1}\varrho_0^\gamma \right)dx.
	\ea\end{equation}
Then we complete the proof of Theorem \ref{the1.3}.
 \end{proof}

\subsection{Vacuum case}
 \begin{proof}[Proof of Theorem \ref{the1.5}]
  For the  vacuum case, we need to mollify $v$ in  both the time and space directions.  With the proof of
  of Theorem \ref{the1.3} in hand, we just need to replace \eqref{ce3.13}
  by the following estimate
\begin{equation}\begin{aligned}
 &\int_0^T\int\phi(t)  \nabla v^{\varepsilon}\varrho[ (v\otimes v)^{\varepsilon} -    v^{\varepsilon}\otimes v^{\varepsilon}] \\
 \leq& C\|\varrho\|_{L^{ \f{p}{p-3}}(L^{ \f{q}{q-3}})}\| (v\otimes v)^{\varepsilon} -    v^{\varepsilon}\otimes v^{\varepsilon}\|_{L^{p/2}(L^{q/2})}\|\nabla v^{\varepsilon}\|_{L^{p}(L^{q})}\\
 \leq& C\varepsilon^{3\alpha-1}\|\varrho\|_{L^{ \f{p}{p-3}}(L^{ \f{q}{q-3}})}\|  v^{\varepsilon}\|^{3}_{ B^{\beta}_{p,\infty}(B^{\alpha}_{q,\infty} )}\\
 \leq& C\varepsilon^{3\alpha-1}\|\varrho\|_{L^{ l}(L^{ k})}\|  v^{\varepsilon}\|^{3}_{ B^{\beta}_{p,\infty}(B^{\alpha}_{q,\infty} )}.
 \end{aligned}\end{equation}
where   the   Constantin-E-Titi type commutators   \eqref{gcet} in Lemma  \ref{lem2.3c} and  the H\"older's inequality are used.
The rest proof is  the same as the one in the last theorem.
  \end{proof}

\section*{Acknowledgement}

 Wang was partially supported by  the National Natural
 Science Foundation of China under grant (No. 11971446, No. 12071113   and  No.  11601492).Ye was partially supported by the National Natural Science Foundation of China  under grant (No.11701145) and China Postdoctoral Science Foundation (No. 2020M672196).
  Yu was partially supported by the
National Natural Science Foundation of China (NNSFC) (No. 11901040), Beijing Natural
Science Foundation (BNSF) (No. 1204030) and Beijing Municipal Education Commission
(KM202011232020).


\begin{thebibliography}{00}
  \bibitem{[ADSW]}
I. Akramov, T. Debiec, J. W. D. Skipper and E. Wiedemann, Energy conservation for the compressible Euler and Navier-Stokes equations with vacuum.  Anal. PDE.  13 (2020),  789--811

\bibitem{[BGSTW]}
 C. Bardos, P. Gwiazda, A. \'Swierczewska-Gwiazda, E. S. Titi and   E. Wiedemann,  Onsager's conjecture in bounded domains for the conservation of entropy and other companion laws. Proc. R. Soc. A, 475 (2019),   18 pp.

\bibitem{[Chae]} D. Chae, On the conserved quantities for the weak solutions of the Euler equations and the quasi-geostrophic equations, Commun. Math. Phys. 266 (2006),  197--210.
\bibitem{[CY]}    R. M. Chen and C. Yu, Onsager's energy conservation for inhomogeneous Euler equations, J. Math. Pures Appl. 131 (2019), 1--16.


\bibitem{[CVY]}  R. M. Chen, A. F. Vasseur and C. Yu. Global ill-posedness for a dense set of initial data to the isentropic
system of gas dynamics. arXiv:2103.04905, 2021.

\bibitem{[CCFS]}
A. Cheskidov and  P. Constantin, S. Friedlander and R. Shvydkoy, Energy conservation and Onsager's conjecture for the Euler equations. Nonlinearity, 21 (2008), 1233--52.
\bibitem{[CL]}
   A. Cheskidov and X. Luo, Energy equality for the Navier-Stokes equations in weak-in-time Onsager spaces. Nonlinearity, 33 (2020), 1388--1403.
 \bibitem{[CDK]}  E. Chiodaroli, C. De Lellis, and O. Kreml. Global ill-posedness of the isentropic system of gas dynamics. Comm. Pure. Appl. Math.,  58(2015), 1157--1190.
\bibitem{[CKMS]}E. Chiodaroli, O. Kreml, V. M\'acha, and S. Schwarzacher. Non-uniqueness of admissible weak solutions to the
compressible Euler equations with smooth initial data. Trans. Amer. Math. Soc.,
374 (2021), 2269--2295.
\bibitem{[CET]} P. Constantin, E. Weinan and E.S. Titi, Onsager's conjecture on the energy conservation for solutions of Euler's equation. Commun. Math. Phys. 165   (1994), 207--209.


    \bibitem{[DE]}
    T. Drivas and G. Eyink. An Onsager singularity theorem for turbulent solutions of compressible Euler
equations.  Commun. Math. Phys., 359 (2018), 733--763.
\bibitem{[PL]}
R. J. DiPerna and P. L. Lions, Ordinary differential equations, transport theory and Sobolev
spaces, Invent. Math., 98 (1989),   511--547.


\bibitem{[DS2]}C. De Lellis and  L  J.  Sz\'ekelyhidi.
On admissibility criteria forweak solutions of the Euler equations. Arch. Ration. Mech. Anal., 195 (2010), 225--260.

\bibitem{[DS0]}
C. De Lellis, L  J.  Sz\'ekelyhidi.
  The Euler equations as a differential inclusion,
Ann. Math. 170  (2009), 1417--1436.

\bibitem{[DS1]}
C. De Lellis, L  J.  Sz\'ekelyhidi.  Dissipative continuous Euler flows. Invent Math. 193(2013 ), 377--407.
\bibitem{[DS3]}C. De Lellis and L  J.  Sz\'ekelyhidi.
Dissipative Euler flows and onsager's conjecture. J Eur Math Soc. 16 (2014), 1467--1505.
\bibitem{[LXX]}
T. Luo, C. Xie and Z. Xin. Non-uniqueness of admissible weak solutions to compressible Euler systems with source terms. Adv. Math., 291(2016) 542--583.
\bibitem{[Feireisl2004]}
E. Feireisl, Dynamics of Viscous Compressible Fluids, Oxford University Press, 2004.
\bibitem{[EGSW]}
     E.   Feireisl, P. Gwiazda, A. Swierczewska-Gwiazda and E. Wiedemann, Regularity and energy conservation for the compressible euler equations. Arch Ration Mech Anal. 223 (2017), 1375--1395.
 \bibitem{[FGJ]}E. Feireisl, S. S. Ghoshal and A. Jana,
   On uniqueness of dissipative solutions to the isentropic Euler system., Commun. Partial Differ. Equ., 44 (2019), 1285--1298.
\bibitem{[GMS]}
P. Gwiazda, M. Mich\'alek and A.\'Swierczewska-Gwiazda, A note on weak solutions of conservation laws and energy/entropy conservation. Arch Ration Mech Anal.,  229 (2018), 1223--1238.

  \bibitem{[Isett]}
P. Isett,
A proof of Onsager's conjecture.
Ann. of Math.  188 (2018),  871--963.
   \bibitem{[LV]}
I. Lacroix-Violet and A. Vasseur, Global weak solutions to the compressible quantum Navier-Stokes equation and its semi-classical limit,
J. Math. Pures Appl. 114 (2018), 191--210.
  \bibitem{[LS]}
T. M. Leslie and R. Shvydkoy,
The energy balance relation for weak solutions of the density-dependent Navier-Stokes equations
J.   Differential Equations. 261 (2016),
  3719--3733.
      \bibitem{[Lions]}
J. L. Lions, Sur la r\'egularit\'e et l'unicit\'e des solutions turbulentes des \'equations de Navier Stokes. Rend. Semin. Mat. Univ. Padova, 30 (1960), 16--23.
  \bibitem{[Lions1]}
 P. L. Lions, Mathematical Topics in Fluid Mechanics, vol. 1. Incompressible Models, Oxford University Press, New York, 1998.
  \bibitem{[Lions2]} P. L. Lions, Mathematical Topics in Fluid Mechanics, vol. 2.  Compressible Models, Oxford University Press, New York, 1998.




 \bibitem{[NNT2]}Q. Nguyen,  P. Nguyen and B. Tang, Onsager's conjecture on the energy conservation for solutions of Euler equations in bounded domains. J. Nonlinear Sci. 29   (2019), 207--213.
  \bibitem{[NNT1]}
Q. Nguyen,  P. Nguyen and B. Tang,  Energy conservation for inhomogeneous incompressible and compressible Euler equations. J. Differential Equations, 269 (2020),  7171--7210.
\bibitem{[Onsager]}
L. Onsager, Statistical hydrodynamics, Nuovo Cim. (Suppl.) 6 (1949), 279--287.

\bibitem{[Simon]}
J. Simon, Compact sets in the space $L^p(0, T; B)$, Ann. Mat. Pura Appl., 146 (1987),    65--96.



 \bibitem{[YWW]}Y. Ye, Y. Wang and W. Wei,
Energy equality in the isentropic compressible Navier-Stokes equations allowing vacuum. arXiv:2108.09425.
 \bibitem{[YWY]}Y. Ye, Y. Wang and H. Yu,
Energy equality for the isentropic compressible Navier-Stokes
equations without upper bound of the density. Preprint. 2021.



\bibitem{[Yu2]}
C. Yu. Energy conservation for the weak solutions of the compressible Navier-Stokes equations.
  Arch. Ration. Mech. Anal. 225 (2017),   1073--1087.







\end{thebibliography}
\end{document}